\documentclass[reqno]{amsart}
\usepackage{hyperref,xcolor,amssymb}
\usepackage[margin=1in]{geometry}
\usepackage[T1]{fontenc}
\usepackage{xcolor}
\usepackage[utf8]{inputenc}

\def\loc{\mathrm{loc}}

\newtheorem{theorem}{Theorem}[section]

\newtheorem{lemma}[theorem]{Lemma}

\theoremstyle{definition}
\newtheorem{definition}[theorem]{Definition}
\newtheorem{remark}[theorem]{Remark}

\numberwithin{equation}{section}

\DeclareMathOperator*{\diam}{diam}
\DeclareMathOperator*{\dist}{dist}
\newcommand{\dd}{\mathrm{d}}

\def\Xint#1{\mathchoice
	{\XXint\displaystyle\textstyle{#1}}%
	{\XXint\textstyle\scriptstyle{#1}}%
	{\XXint\scriptstyle\scriptscriptstyle{#1}}%
	{\XXint\scriptscriptstyle\scriptscriptstyle{#1}}%
	\!\int}
\def\XXint#1#2#3{{\setbox0=\hbox{$#1{#2#3}{\int}$ }
		\vcenter{\hbox{$#2#3$ }}\kern-.6\wd0}}

\def\dashint{\Xint-}

\title{Gradient higher integrability for double phase problems on metric measure spaces}
\author{Juha Kinnunen, Antonella Nastasi, Cintia Pacchiano Camacho}
\date{July 2022}

\begin{document}

\address{J.K.: Department of Mathematics, Aalto University, P.O. Box 11100, FI-00076 Aalto, Finland}
\email{juha.k.kinnunen@aalto.fi}

\address{A.N.: Department of Engineering, University of Palermo, Viale delle Scienze, 90128, Palermo, Italy}
\email{antonella.nastasi@unipa.it}

\address{C.P.C.: Department of Mathematics and Statistics, University of Calgary, 2500 University Dr. NW, Calgary, AB T2X 3B5, Canada}
\email{cintia.pacchiano@ucalgary.ca}


\maketitle

\begin{abstract}
We study local and global higher integrability properties for quasiminimizers of a class of double phase integrals characterized by nonstandard growth conditions. We work purely on a variational level in the setting of a metric measure space with a doubling measure and a Poincar\'e inequality. 
The main novelty is an intrinsic approach to double phase Sobolev-Poincar\'e inequalities. 
\end{abstract}
\section{Introduction}

Assume that $(X,d,\mu)$ is a complete metric measure space endowed with a metric $d$ and a doubling measure $\mu$ and supporting a weak $(1,p)$-Poincar\'e inequality.
Let  $\Omega$ be an open subset of $X$. 
This paper discusses regularity properties of the minimal $p$-weak upper gradient of quasiminimizers of the double phase integral 
\begin{equation}\label{J}
\int_{\Omega}H(x,g_{u}) \,\dd \mu
=\int_{\Omega}(g_{u}^p + a(x) g_{u}^q) \,\dd \mu,
\end{equation} 
with
\[
1<\frac{q}{p}\le1+\frac{\alpha}{Q}, 
\quad 0<\alpha\le1, 
\quad Q=\log_2C_D,
\] 
where $p>1$ and $C_D$ is the doubling constant of the measure. 
Observe that $Q$ is a notion of dimension related to the measure $\mu$. 
For example, in the Euclidean $n$-space with the Lebesgue measure we have $Q=n$.
The double phase functional in \eqref{J} is denoted by
\[
H(x,z)= |z|^p+a(x)|z|^q, 
\quad x \in \Omega,\quad z\in \mathbb{R}.
\]
The nonnegative coefficient function $a$ is assumed to be $\alpha$-H\"older continuous with respect to a quasi-distance related to the underlying measure $\mu$, see \eqref{aalpha} below for the precise definition.
This reduces to the standard H\"older continuity with the exponent $\alpha$, if the measure is $Q$-Ahlfors--David regular.

The main feature of the functional \eqref{J} is that it switches between
two different types of growth conditions determined by the coefficient function $a$.
When $a(x)=0$, the variational integral in \eqref{J} reduces to the familiar problem with $p$-growth and when $a(x)\ge c>0$ we have the $(p,q)$-problem.
Thus the zero set $\{a(x)=0\}$ plays a decisive role in \eqref{J}.
The main advantage of the notion of quasiminimizer of \eqref{J} is that it simultaneously covers a large class of problems 
where the variational integrand $F:\Omega\times\mathbb R\times \mathbb R\to\mathbb R$ satisfies the Carath\'eodory conditions and
\[
\lambda H(x,z)\le F(x,u,z)\le\Lambda H(x,z),
\quad 0<\lambda<\Lambda<\infty,
\]
for every $x \in \Omega$ and $u,z\in \mathbb{R}$.
For quasiminimizers with the $p$-growth on Euclidean spaces, see \cite{DT,GG1,GG2}, and on metric measure spaces, see \cite{BB, BMS, KMM, KM, KS}. 
This paper extends the theory for quasiminimizers on metric measure space to the double phase problems.

The natural function space for a local quasiminimizer of \eqref{J} is $u\in N^{1,1}_{\loc}({\Omega})$ with $H(\cdot,g_u)\in L^1_{\loc}(\Omega)$, 
where $N^{1,1}$ denotes the Newtonian--Sobolev space on a metric measure space, see \cite{BB}, \cite{HKS} and \cite{S}. 
We show that if $u$ is a local quasiminimizer of \eqref{J}, then $H(x,g_u)$ is locally integrable to a slightly higher power than one, see Theorem \ref{mainlocal}. 
We also discuss the corresponding question up to the boundary for quasiminiminimizers with boundary values, see Theorem \ref{mainglobal}.
For this kind of local higher integrability results in the Euclidean case, see \cite{B,Ge,GG1,GG2,GM,M,ME}.
For the corresponding global results, we refer to \cite{Gr,KilKos}.
For results with functionals of the type \eqref{J} in the Euclidean setting, we refer to \cite{CM1,CM,DF, DFM, ELM, Ma1, Ma2, Ma3, Min, MinR}. 
Higher integrability questions for variational problems on metric measure spaces have been studied in \cite{FH1,FH2,FHJM, H1,H2,MZG, MM, MMPP, NPC2}.
Our work shows that the corresponding theory can be developed for double phase problems on metric measure spaces. 
The argument is based on energy estimates, double phase Sobolev-Poincar\'e inequalities and a self-improving property of reverse H\"older inequalities.

\section{Preliminaries}\label{Sec2}
Throughout the paper, positive constants are denoted by $C$ and 
the dependencies on parameters are listed in the parentheses.
We assume that $(X, d, \mu)$ is a complete metric measure space with a metric $d$ and a Borel regular measure $\mu$. 
The measure $\mu$ is assumed to be doubling, that is, there exists a constant $C_D \geq 1$ such that 
\begin{equation}\label{doubling}
0<\mu(B_{2r})\leq C_D \mu(B_r)<\infty,
\end{equation} 
for every ball $B_r$ in $X$.
Here $B_r=B_r(x)=\{x\in X:d(y,x)<r\}$ is an open ball with the center $x\in X$ and the radius $0<r<\infty$. 
The following result gives a notion of dimension related to a doubling measure.

\begin{lemma}[\cite{BB}, Lemma 3.3]\label{lemm3.3}
	Let $(X, d, \mu)$ be a metric measure space with a doubling measure $\mu$. Then 
	\begin{equation}\label{s}
		\frac{\mu(B_{r}(y))}{\mu(B_R(x))}\geq C\left(\frac{r}{R}\right)^Q,
	\end{equation}
	for every $0<r\le R<\infty$, $x \in X$ and $y \in B_R(x)$. Here $Q=\log_2C_D$ and $C=C_D^{-2}$.
\end{lemma}

A complete metric measure space with a doubling measure is proper, that is, closed and bounded subsets are compact, see \cite[Proposition 3.1]{BB}.
We discuss the notion of upper gradient as a way to generalize modulus of the gradient in the Euclidean case to the metric setting. 
For further details, we refer to the book by  Bj\"{o}rn and Bj\"{o}rn \cite{BB}.

\begin{definition}
	A nonnegative Borel function $g$ is said to be an upper gradient of function $u: X \to [-\infty,\infty]$ if, for all paths $\gamma$ connecting $x$ and $y$, we have 
	\begin{equation*}
		|u(x)-u(y)|\leq \int_{\gamma}g\, \dd s, 
	\end{equation*}
	whenever $u(x)$ and $u(y)$ are both finite and $\int_{\gamma}g \, \dd s= \infty$ otherwise. Here $x$ and $y$ are the endpoints of $\gamma$.
	Moreover, if a nonnegative measurable function $g$ satisfies the inequality above for $p$-almost every path, 
	that is, with the exception of a path family of zero $p$-modulus, then $g$ is called a $p$-weak upper gradient of $u$.
\end{definition}

For $1\leq p<\infty$ and an open set $\Omega\subset X$, let
\[
\Vert u\Vert_{N^{1,p}(\Omega)}=\Vert u\Vert_{L^{p}(\Omega)}+\inf\Vert g\Vert_{L^{p}(\Omega)},
\]
where the infimum is taken over all upper gradients $g$ of $u$.
Consider the collection of functions $u\in L^p(\Omega)$ with an upper gradient $g\in L^p(\Omega)$ and let
\begin{equation*}
	\widetilde{N}^{1,p}(\Omega)
	=\lbrace u:\Vert u\Vert_{N^{1,p}(\Omega)}<\infty\rbrace.
\end{equation*}
The Newtonian space is defined by
\[
N^{1,p}(\Omega)=\lbrace u:\Vert u\Vert_{N^{1,p}(\Omega)}<\infty\rbrace/\sim,
\]
where $u\sim v$ if and only if $\Vert u-v\Vert_{N^{1,p}(\Omega)}=0$.

The corresponding local Newtonian space is defined by $u\in N^{1,p}_{\loc}(\Omega)$ if
$u\in N^{1,p}(\Omega')$ for all $\Omega'\Subset \Omega$, see \cite[Proposition 2.29]{BB},
where $\Omega'\Subset \Omega$ means that $\overline{\Omega'}$ is a compact subset of $\Omega$.
If $u$ has an upper gradient $g\in L^p(\Omega)$, there exists a unique minimal $p$-weak upper gradient $g_u\in L^p(\Omega)$ with
$g_u\le g$ $\mu$-almost everywhere for all $p$-weak upper gradients $g\in L^p(\Omega)$ of $u$, see  \cite[Theorem 2.5]{BB}.
Moreover, the minimal $p$-weak upper gradient is unique up to sets of measure zero.
For $u\in N^{1,p}(\Omega)$ we have
\[
\Vert u\Vert_{N^{1,p}(\Omega)}=\Vert u\Vert_{L^{p}(\Omega)}+\Vert g_u\Vert_{L^{p}(\Omega)},
\]
where $g_u$ is the minimal $p$-weak upper gradient of $u$.
The main advantage is that $p$-weak upper gradients behave better under $L^p$-convergence than upper gradients, see \cite[Proposition 2.2]{BB}.
However, the difference is relatively small, since every $p$-weak upper gradient can be approximated by a sequence of upper gradients in $L^p$, see \cite[Lemma 1.46]{BB}.
This implies that the $N^{1,p}$-norm above remains the same if the infimum is taken over upper gradients instead of $p$-weak upper gradients.

Let $\Omega$ be an open subset of $X$. We define $N^{1,q}_0(\Omega)$ to be the set of functions $u\in N^{1,q}(X)$ that are zero on $X\setminus\Omega$ $\mu$-a.e. The space $N_0^{1,q}(\Omega)$ is equipped with the norm $\Vert\cdot\Vert_{N^{1,q}}$. Note also that if $\mu(X\setminus\Omega) = 0$, then $N^{1,q}_0(\Omega)=N^{1,q}(X)$. We shall therefore always assume that $\mu(X \setminus \Omega) > 0$.

The integral average is denoted by
\[
u_{B}=\dashint_{B} u\, \dd\mu
=\frac{1}{\mu(B)}\int_{B}u\,\dd\mu.
\]
We assume that $X$ supports the following Poincar\'{e} inequality.

\begin{definition}
Let $1\le p<\infty$. 
A metric measure space $(X,d,\mu)$ supports a weak $(1, p)$-Poincar\'{e} inequality if there exist a constant $C_{PI}$ and a dilation factor $\lambda \geq 1 $ such that 
\[
	\dashint_{B_r} |u-u_{B_r}|\, \dd\mu
	\leq C_{PI} r \left(\dashint_{B_{\lambda r}}g_u^p \, \dd\mu\right)^{\frac{1}{p}},
\]
for every ball $B_r$ in $X$ and for every $u\in L^1_{\loc}(X)$.
\end{definition}

As shown in \cite[ Theorem 1.0.1 ]{KZ} by Keith and Zhong, see also \cite[Theorem 4.30]{BB}, the Poincar\'{e} inequality is a self-improving property.

\begin{theorem}\label{kz}
	Let $(X, d, \mu)$ be a complete metric measure space with a doubling measure $\mu$ and a weak $(1,p)$-Poincar\'{e} inequality with $p>1$.
	Then there exists $\varepsilon>0$ such that $X$ supports a weak $(1, q)$-Poincar\'{e} inequality for every $q>p-\varepsilon$. 
	Here, $\varepsilon$ and the constants associated with the $(1, q)$-Poincar\'e inequality depend only on $C_D$, $C_{PI}$ and $p$.
\end{theorem}

The following result shows that the Poincar\'e inequality implies a Sobolev--Poincar\'e inequality, 
see \cite[Theorem 4.21 and Corollary 4.26]{BB}.

\begin{theorem}\label{sstars}
	Assume that $\mu$ is a doubling measure and $X$ supports a weak $(1,p)$-Poincar\'{e} inequality and let $Q=\log_2C_D$ be as in \eqref{s}.  
	Let $1\le p^*\le\frac{Qp}{Q-p}$ for $1\le p<Q$ and $1\leq p^*<\infty$ for $Q\leq p<\infty$.
	Then $X$ supports a weak $(p^*,p)$-Poincar\'{e} inequality,
	that is, there exist a constant $C=C(C_D,C_{PI},p)$ such that
	\[
		\left(\dashint_{B_r} |u-u_{B_r}|^{p^*} \,\dd\mu\right)^{\frac{1}{p^*}}
		\leq Cr\left(\dashint_{B_{2\lambda r}}g_u^p \,\dd\mu\right)^{\frac{1}{p}},
	\]
	for every ball $B_r$ in $X$ and every $u \in L^1_{\loc}(X)$. 
\end{theorem}

The following notation and assumptions will be used throughout the paper.
For the coefficient function $a:X\to[0,\infty)$ in \eqref{J}, we assume that there exists $\alpha$, $0<\alpha\le1$, such that
\begin{equation}\label{aalpha}
	[a]_{\alpha}= \sup_{x,y \in \Omega, x\neq y} \dfrac{|a(x)-a(y)|}{\delta_{\mu}(x,y)^{\alpha}}<\infty,
\end{equation}
where $\delta_{\mu}$ is a quasi-distance given by 
\[
\delta_{\mu}(x,y)=\bigl(\mu(B_{d(x,y)}(x))+\mu(B_{d(x,y)}(y))\bigr)^{1/Q},\quad x,y \in X,\, x\ne y.
\]
Here $Q=\log_2C_D$ is as in \eqref{s} and we set $\delta_{\mu}(x,x)=0$.

\begin{remark}
A measure is called Ahlfors--David regular, if there exist constants $0<C_1\le C_2<\infty$ such that
\begin{equation}\label{ahlfors}
C_1r^Q\le\mu(B_r(x))\le C_2r^Q,
\end{equation}
for every $x\in X$ and $0<r\le\diam(X)$.
If the measure $\mu$ is Ahlfors--David regular, then $\delta_{\mu}(x,y)\approx d(x,y)$ for every $x,y$ and, consequently,
$[a]_\alpha<\infty$ if and only if $a$ is H\"older continuous with the exponent $\alpha$.
\end{remark}

We assume that 
\begin{equation}\label{pqcond}
1<\frac{q}{p}\le1+\frac{\alpha}{Q},
\end{equation}
where $p>1$, $\alpha$ is as in \eqref{aalpha} and  $Q=\log_2C_D$ is as in \eqref{s}.

By \eqref{pqcond}, Theorem \ref{kz} and Theorem \ref{sstars} there exists $s=s(C_D,p,q)$, with $1<s<p<q<s^*$, such that $X$ supports a $(s^*,s)$-Poincar\'e inequality, that is,
\begin{equation}\label{spoincare1}
\left(\dashint_{B_r} \left|u-u_{B_r}\right|^{s^*}\,\dd\mu\right)^{\frac{1}{s^*}} 
\leq Cr\left(\dashint_{B_{2\lambda r}} g_u^{s}\,\dd\mu\right)^{\frac{1}{s}},
\end{equation}
for every ball $B_r$ in $X$ and every $u \in L^1_{\loc}(X)$ with $C=C(C_D,C_{PI},\lambda,p,q)$.
We keep track on dependencies and denote
\[
C(\mathrm{data})=C(C_D,C_{PI},\lambda,p,q,K,\alpha,[a]_\alpha).
\]
Here $K$ is the quasimimizing constant in Definition \ref{lqm} below.
By the structure of a double phase functional we have $g_u\in L^p(\Omega)$. 
 However, we cannot conclude that $g_u\in L^q(\Omega)$, since the function $a$ may be zero on a subset of $\Omega$.
Next we discuss the definition of a local quasiminimizer.

\begin{definition}\label{lqm}
	A function $u\in N^{1,1}_{\loc}({\Omega})$ with $H(\cdot,g_u)\in L^1_{\loc}(\Omega)$ is a local quasiminimizer on $\Omega$, if there exists a constant $K\geq 1$ such that 
	\[
		\int_{\Omega'\cap\{u\ne v\}}H(x,g_u)\,\dd \mu
		\leq K\int_{\Omega'\cap\{u\ne v\}}H(x,g_v)\,\dd \mu,
\]
	for every open subset $\Omega'\Subset\Omega$ and for every function $v\in N^{1,1}(\Omega')$ with $u-v\in  N^{1,1}_0(\Omega')$.
\end{definition}

Then we give a definition of quasiminimizers with boundary values.

\begin{definition}\label{qmbv}
Let $w\in N^{1,1}(\Omega)$ with $H(\cdot,g_w)\in L^1(\Omega)$.
	A function $u\in N^{1,1}({\Omega})$  with $H(\cdot,g_u)\in L^1(\Omega)$ is a quasiminimizer on $\Omega$ with the boundary values $w$,  
	if $u-w\in N^{1,1}_0(\Omega)$ and there exists a constant $K\geq 1$ such that 
	\[
		\int_{\Omega'\cap\{u\ne v\}}H(x,g_u)\, \dd \mu
		\leq K\int_{\Omega'\cap\{u\ne v\}}H(x,g_v)\, \dd \mu,
	\]
	for every open subset $\Omega'\subset\Omega$ and for every function $v \in N^{1,1}(\Omega')$ with $u-v\in  N^{1,1}_0(\Omega')$.
\end{definition}

The main difference in the definitions above is that the assumption $u\in N^{1,1}_{\loc}({\Omega})$ with $H(\cdot,g_u)\in L^1_{\loc}(\Omega)$ in the local case 
is replaced with $u\in N^{1,1}({\Omega})$  with $H(\cdot,g_u)\in L^1(\Omega)$.
It is obvious that a quasiminimizer with boundary values is a local quasiminimizer. 

We state a local energy estimate for the double phase problem.

\begin{lemma}\label{DeGiorgiLemma1}
Assume that  $u\in N^{1,1}_{\loc}({\Omega})$ with $H(\cdot,g_u)\in L^1_{\loc}(\Omega)$  is a local quasiminimizer in $\Omega$ and let $B_r\subset B_R\Subset\Omega$ be concentric balls. Then there exists a constant $C= C(K,q)$ such that  
\[
\int_{B_r}H(x,g_u)\, \dd \mu
\leq C\int_{B_R} H\left(x,\frac{u-u_{B_R}}{R-r}\right)\,\dd\mu.
\]
\end{lemma}

\begin{proof}
Let $\eta$ be a $(R-r)^{-1}$-Lipschitz cutoff function such that $0\leq \eta \leq1$, $\eta=1$ on $B_r$ and $\eta=0$ in $X\setminus B_R$.
Let $v=u- \eta(u-u_{B_R})$.
By the Leibniz rule for the upper gradients (\cite{BB}, Lemma 2.18]), we have
\begin{equation*}
g_{v}\leq |u-u_{B_R}|g_{\eta}+(1-\eta) g_u
\le\dfrac{|u-u_{B_R}|}{R-r}+(1-\chi_{B_{r}})g_{u}.
\end{equation*}
Since $u$ is a local quasiminimizer and $u-v\in N^{1,1}_0(B_R)$, by Definition \ref{lqm} we obtain
\begin{equation}\label{5.7}
\begin{split}
\int_{B_{r}}H(x,g_u)\, \dd \mu  
&\leq\int_{B_{R}}H(x,g_u)\, \dd\mu 
\leq K\int_{B_{R}}H(x,g_v)\, \dd\mu \\
& \leq 2^qK\left(\int_{B_{R}}H\left(x,\frac{u-u_{B_R}}{R-r}\right)\,\dd\mu
+\int_{B_{R}\setminus B_{r}}H(x,g_u)\, \dd \mu\right).
\end{split}
\end{equation}	
By adding $K 2^{q}\int_{B_{r}}H(x,g_u)\, \dd \mu$ to the both sides of (\ref{5.7}), we get 
\[
(1+K 2^{q}) \int_{B_{r}}H(x,g_u)\, \dd \mu   
\leq K2^{q}\left(\int_{B_{R}}H\left(x,\frac{u-u_{B_R}}{R-r}\right)\, \dd\mu
+ \int_{B_{R}}H(x,g_u)\, \dd \mu\right).
\]
This implies
\begin{align*}
\int_{B_{r}}H(x,g_u)\, \dd \mu   
&\leq \theta\left(\int_{B_{R}}H\left(x,\frac{u-u_{B_R}}{R-r}\right)\, \dd\mu
+ \int_{B_{R}}H(x,g_u)\, \dd \mu\right)\\
&\le(R-r)^{-p}\int_{B_{R}}|u-u_{B_R}|^p\, \dd \mu
+(R-r)^{-q}\int_{B_{R}}a|u-u_{B_R}|^q\, \dd \mu
+\theta\int_{B_{R}}H(x,g_u)\, \dd \mu,
\end{align*}
with $\theta= \frac{K 2^{q}}{1+K 2^{q}}<1$.
We apply a standard iteration lemma, see \cite[Lemma 6.1]{G}, to obtain
\[
\int_{B_{r}}H(x,g_u)\, \dd \mu
\leq C\int_{B_{R}}H\left(x,\frac{u-u_{B_R}}{R-r}\right)\, \dd\mu,
\]
where $C=C(q,K)$.
\end{proof}

Next we discuss a global energy estimate for quasiminimizers with boundary values.

\begin{lemma}\label{DeGiorgiLemma2}
Let $w\in N^{1,1}(\Omega)$ with $H(\cdot,g_w)\in L^1(\Omega)$.
Assume that $u\in N^{1,1}({\Omega})$  with $H(\cdot,g_u)\in L^1(\Omega)$ is a quasiminimizer on $\Omega$ with $u-w\in N^{1,1}_0(\Omega)$ and 
let $B_r\subset B_R$ be concentric balls. Then there exists a constant $C=C(K,q)$ such that
 \begin{equation*}
  \int_{B_{r}\cap\Omega}H(x, g_u)\,\dd\mu
  \leq C\left( \int_{B_R\cap\Omega}H\left(x, \frac{u-w}{R-r}\right)\,\dd\mu
  +\int_{B_R\cap\Omega}H(x, g_w)\,\dd\mu\right).
\end{equation*}
\end{lemma}

\begin{proof}
Let $\eta$ be a $(R-r)^{-1}$-Lipschitz cutoff function such that $0\leq \eta \leq1$, $\eta=1$ on $B_r$ and $\eta=0$ in $X\setminus B_R$.
Let $v=u-\eta(u-w)$.
Then $\eta(u-w)\in N^{1,1}_0(B_R\cap\Omega)$ and thus $v-u\in N^{1,1}_0(B_R\cap\Omega)$.
By Definition \ref{qmbv}, we obtain
$$
\int_{B_R\cap\Omega}H(x,g_u)\,\dd\mu\leq K\int_{B_R\cap\Omega}H(x,g_v)\,\dd\mu,
$$
where $v=u+\eta(w-u)$.
Since 
\[
g_v\leq\vert u-w\vert g_\eta+(1-\eta)g_u+\eta g_w
\le\dfrac{\vert u-w\vert}{R-r}+(1-\chi_{B_{r}})g_{u}+g_w,
\] 
we obtain
\begin{align*}
\int_{B_{r}\cap\Omega}H(x,g_u)\, \dd \mu  
&\leq\int_{B_{R}\cap\Omega}H(x,g_u)\, \dd\mu 
\leq K\int_{B_{R}\cap\Omega}H(x,g_v)\, \dd\mu \\
& \leq 3^qK\left(\int_{B_{R}\cap\Omega}H\left(x,\frac{u-w}{R-r}\right)\,\dd\mu
+\int_{(B_{R}\setminus B_{r})\cap\Omega}H(x,g_u)\, \dd \mu
+\int_{B_{R}\cap\Omega}H(x,g_w)\, \dd \mu\right).
\end{align*}
By filling the hole and iterating as in the proof of Lemma \ref{DeGiorgiLemma1}, we arrive at
\[
  \int_{B_{r}\cap\Omega}H(x, g_u)\,\dd\mu
  \leq C\left( \int_{B_R\cap\Omega}H\left(x, \frac{u-w}{R-r}\right)\,\dd\mu
  +\int_{B_R\cap\Omega}H(x, g_w)\,\dd\mu\right),
\]
where $C=C(q,K)$.
\end{proof}

\section{Double phase Sobolev--Poincar\'e inequalities}\label{Double Phase SP ineq}

This section discusses double phase Sobolev--Poincar\'e inequalities. 
We consider interior and boundary estimates separetely.
We begin with interior estimates.

\begin{lemma}\label{Lemma2}
Assume that $u\in N^{1,1}_{\loc}({\Omega})$ with $H(\cdot,g_u)\in L^1_{\loc}(\Omega)$. Let $a_0=\inf_{x\in B_{2\lambda r}}a(x)$. 
Then there exist a constant $C= C(\mathrm{data})$ and exponents $0<d_2<1\leq d_1<\infty$, with $d_1= d_1(C_D,p,q)$ and $d_2= d_2(C_D,p,q)$, such that
\begin{equation}
\left(\dashint_{B_r}\left( \left|\dfrac{u-u_{B_r}}{r}\right|^p+a_0\left|\dfrac{u-u_{B_r}}{r}\right|^q \right)^{d_1}\,\dd\mu\right)^{\frac{1}{d_1}} 
\leq C \left(\dashint_{B_{2\lambda r}} \left( g_u^{p}+ a_0 g_u^q\right)^{d_2}\,\dd\mu\right)^{\frac{1}{d_2}},	
\end{equation} whenever $B_{2\lambda r} \Subset\Omega$.
\end{lemma}

\begin{proof}
By \eqref{spoincare1} there exists $s$, with $1<s<p<q<s^*$, such that 
\begin{equation}\label{spoincare}
\left(\dashint_{B_r} \left|\frac{u-u_{B_r}}{r}\right|^{s^*}\,\dd\mu\right)^{\frac{1}{s^*}} 
\leq C\left(\dashint_{B_{2\lambda r}} g_u^{s}\,\dd\mu\right)^{\frac{1}{s}}.
\end{equation}
Let $\frac{s}{p}<d_2<1$ and $1\le d_1<\frac{s^*}{q}$.
Since $pd_1<qd_1<s^*$ and $s<pd_2<qd_2$, by H\"older's inequality, we have
\begin{equation}\label{poincared1}
\left(	\dashint_{B_r} \left|\dfrac{u-u_{B_r}}{r}\right|^{p d_1}\, \dd\mu\right)^{\frac{1}{pd_1}} 
\leq C \left(\dashint_{B_{2\lambda r}} g_u^{p d_2}\, \dd\mu\right)^{\frac{1}{pd_2}},
\end{equation}
and 
\begin{equation*}
\left(	\dashint_{B_r} \left|\dfrac{u-u_{B_r}}{r}\right|^{q d_1}\,\dd\mu\right)^{\frac{1}{qd_1}}
\leq C \left(\dashint_{B_{2\lambda r}} g_u^{q d_2}\,\dd\mu\right)^{\frac{1}{qd_2}}.
\end{equation*}
It follows that
\begin{align*}
&\left(\dashint_{B_r}\left(\left\vert\frac{u-u_{B_r}}{r}\right\vert^p+a_0\left\vert\frac{u-u_{B_r}}{r}\right\vert^q\right)^{d_1}\,\dd\mu\right)^{\frac{1}{d_1}}\\
&\qquad\leq \left(\dashint_{B_r}\left|\dfrac{u-u_{B_r}}{r}\right|^{p d_1}\,\dd\mu\right)^{\frac{1}{d_1}}
+a_0\left(\dashint_{B_r}  \left|\dfrac{u-u_{B_r}}{r}\right|^{q d_1}\,\dd\mu  \right)^{\frac{1}{d_1}}\\
&\qquad\leq C \left( \left(\dashint_{B_{2\lambda r}}  g_u^{p d_2}\,\dd\mu\right)^{\frac{1}{d_2}}+ a_0  \left(\dashint_{B_{2\lambda r}} g_u^{q d_2}\,\dd\mu\right)^{\frac{1}{d_2}} \right)\\
&\qquad\leq C \left(\dashint_{B_{2\lambda r}}  g_u^{p d_2}\,\dd\mu+\dashint_{B_{2\lambda r}} \left(a_0g_u^q\right)^{d_2}\,\dd\mu\right)^{\frac{1}{d_2}}\\ 
&\qquad\leq C\left(\dashint_{B_{2\lambda r}}  \left(g_u^{p}+ a_0g_u^{q}\right)^{d_2}\,\dd\mu\right)^{\frac{1}{d_2}},
\end{align*}
where $C= C(C_D,C_{PI},\lambda,p,q)$.
Observe that all integrals are finite, since
\begin{align*}
\left(\dashint_{B_{2\lambda r}}\left(g_u^{p}+ a_0g_u^{q}\right)^{d_2}\,\dd\mu\right)^{\frac{1}{d_2}}
&\le\left(\dashint_{B_{2\lambda r}}\left(g_u^{p}+ a(x)g_u^{q}\right)^{d_2}\,\dd\mu\right)^{\frac{1}{d_2}}\\
&\le\dashint_{B_{2\lambda r}}\left(g_u^{p}+ a(x)g_u^{q}\right)\,\dd\mu\\
&=\dashint_{B_{2\lambda r}}H(x,g_u)\,\dd\mu<\infty.
\end{align*}
\end{proof}

Then we consider an interior double phase Sobolev--Poincar\'e inequality.

\begin{lemma}\label{Lemma3}
Assume that $u\in N^{1,1}_{\loc}({\Omega})$ with $H(\cdot,g_u)\in L^1_{\loc}(\Omega)$.
Then there exists a constant $C=C(\mathrm{data})$ and exponents $0<d_2<1\leq d_1<\infty$,  with $d_1= d_1(\mathrm{data})$ and $d_2= d_2(\mathrm{data})$, such that 
\begin{equation*}
\left(\dashint_{B_r}H\left(x,\frac{u-u_{B_r}}{r}\right)^{d_1}\,\dd\mu\right)^{\frac{1}{d_1}}
\leq C\left(1+\Vert g_u\Vert^{q-p}_{L^p(B_{2\lambda r})}\mu(B_{2\lambda r})^{\frac{\alpha}{Q}-\frac{q-p}{p}}\right)
\left(\dashint_{B_{2\lambda r}} H(x,g_u)^{d_2}\,\dd\mu\right)^{\frac{1}{d_2}},
\end{equation*}
whenever $B_{2\lambda r}\Subset\Omega$. 
\end{lemma}

\begin{proof}
First assume that
\begin{equation}\label{1C}
a_0=\inf_{x\in B_{2\lambda r}}a(x)>2[a]_{\alpha}(2C_D^2\mu(B_{2\lambda r}))^{\alpha/Q},
\end{equation}
where $Q=\log_2C_D$ is as in Lemma \ref{lemm3.3}.
Note that, for every $x,y\in B_{2\lambda r}$, we have
\begin{align*}
\delta_{\mu}(x,y)
&=\left(\mu(B_{d(x,y)}(x))+\mu(B_{d(x,y)}(y))\right)^{1/Q}\\
&\leq \left(\mu(B_{4\lambda r}(x))+\mu(B_{4\lambda r}(y))\right)^{1/Q}\\
&\leq \left(2\mu (B_{6\lambda r})\right)^{1/Q}
\leq (2C_D^2\mu(B_{2\lambda r}))^{1/Q}.
\end{align*}
By \eqref{1C} we obtain
\begin{align*}
2a_0&= 2a(x)-2\left(a(x)-a_0\right)\geq a(x)+a_0-2\left(a(x)-a_0\right)\\ 
&\geq a(x)+2[a]_{\alpha}(2C_D^2\mu(B_{2\lambda r}))^{\alpha/Q}-2\left(a(x)-a_0\right)\\
&\geq a(x)+2\sup_{\substack{x,y \in B_{2\lambda r}\\ x\neq y}} \frac{|a(x)-a(y)|}{\delta_{\mu}(x,y)^{\alpha}}(2C_D^2\mu(B_{2\lambda r}))^{\alpha/Q}-2\left(a(x)-a_0\right)\\
&\geq a(x)+2\sup_{x,y \in B_{2\lambda r}}|a(x)-a(y)|-2\left(a(x)-a_0\right)\\
&\geq a(x)+2\sup_{x,y \in B_{2\lambda r}}(a(x)-a(y))-2\left(a(x)-a_0\right)\\
&\geq a(x)+2a(x)-2\inf_{y \in B_{2\lambda r}}a(y)-2\left(a(x)-a_0\right)
=a(x),
\end{align*}
for every $x \in B_{2\lambda r}$.
On the other hand, we have $a(x)\geq\inf_{x\in B_{2\lambda r}} a(x)=a_0$ for every $x\in B_{2\lambda r}$.
This implies that $a_0\le a(x)\le 2a_0$ for every $x \in B_{2\lambda r}$.
By Lemma \ref{Lemma2}, we conclude that
\begin{equation}\label{2C}
\begin{split}
\left(\dashint_{B_r} H\left(x,\frac{u-u_{B_r}}{r}\right)^{d_1}\,\dd\mu\right)^{\frac{1}{d_1}}
&= \left(\dashint_{B_r}\left(\left|\dfrac{u-u_{B_r}}{r}\right|^p+ a(x)  \left|\dfrac{u-u_{B_r}}{r}\right|^q \right)^{d_1} \,\dd\mu\right)^{\frac{1}{d_1}}\\
&\leq C\left(\dashint_{B_r}\left(\left|\dfrac{u-u_{B_r}}{r}\right|^p+ a_0  \left|\dfrac{u-u_{B_r}}{r}\right|^q \right)^{d_1} \,\dd\mu \right)^{\frac{1}{d_1}}\\
&\leq C \left(\dashint_{B_{2\lambda r}} \left( g_u^{p}+ a_0 g_u^q\right)^{d_2}\,\dd\mu\right)^{\frac{1}{d_2}}\\
&\leq C \left(\dashint_{B_{2\lambda r}} \left( g_u^{p}+ a(x) g_u^q\right)^{d_2}\,\dd\mu\right)^{\frac{1}{d_2}}\\
&\leq C \left(\dashint_{B_{2\lambda r}} H(x,g_u)^{d_2} \,\dd\mu\right)^{\frac{1}{d_2}},
\end{split}
\end{equation}
where $C= C(\mathrm{data})$, $d_1= d_1(\mathrm{data})$ and $d_2= d_2(\mathrm{data})$ with $0<d_2<1\leq d_1<\infty$.

Next we consider the case which is complementary to \eqref{1C}, that is, 
\begin{equation}\label{3C}
a_0=\inf_{x\in B_{2\lambda r}}a(x)\leq 2[a]_{\alpha}(2C_D^2\mu(B_{2\lambda r}))^{\alpha/Q}.
\end{equation}
Notice that, for every $x\in  B_{2\lambda r}$ and $y \in B_r$, with $y\neq x$, we have
\begin{align}\label{notice}
    a(y)-a(x)&\leq \vert a(x)-a(y)\vert=\frac{\vert a(x)-a(y)\vert}{\delta_{\mu}(x,y)^{\alpha}}\delta_{\mu}(x,y)^{\alpha}\leq  [a]_{\alpha}\delta_{\mu}(x,y)^{\alpha}.
\end{align}
Note that, for every $x\in  B_{2\lambda r}$ and $y \in B_r$, with $y\neq x$, we have
\begin{align*}
\delta_{\mu}(x,y)
&=\left(\mu(B_{d(x,y)}(x))+\mu(B_{d(x,y)}(y))\right)^{1/Q}\\
&\leq \left(\mu(B_{3\lambda r}(x))+\mu(B_{3\lambda r}(y))\right)^{1/Q}\\
&\leq \left(2\mu (B_{5\lambda r})\right)^{1/Q}
\leq C\mu(B_{2\lambda r})^{1/Q},
\end{align*}
where $C=C(C_D)$.
By \eqref{notice}, we get
\[
a(y)\leq a(x)+C[a]_{\alpha}\mu(B_{2\lambda r})^{\alpha/Q},
\]
where $C=C(C_D,\alpha)$.
By taking infimum over all $x\in 2\lambda B_r$, we obtain
\begin{align*}
 a(y)
 &\leq \inf_{x\in B_{2\lambda r}}a(x)+C[a]_{\alpha}\mu(B_{2\lambda r})^{\alpha/Q}\\
 &\leq 2[a]_{\alpha}(2C_D^2\mu(B_{2\lambda r}))^{\alpha/Q}+C[a]_{\alpha}\mu(B_{2\lambda r})^{\alpha/Q}\\\
 &= C[a]_{\alpha}\mu(B_{2\lambda r})^{\alpha/Q},
\end{align*}
where $C=C(C_D,\alpha)$.
By taking supremum over $y\in B_r$, we conclude that
\[
\sup_{y\in B_r}a(y)
\leq C[a]_{\alpha}\mu(B_{2\lambda r})^{\alpha/Q}.
\]
It follows that
\begin{equation}\label{6lemma3}
\begin{split}
	&\left(\dashint_{B_r} H\left(x,\frac{u-u_{B_r}}{r}\right)^{d_1}\,\dd\mu\right)^{\frac{1}{d_1}}
	= \left(\dashint_{B_r}\left(\left|\dfrac{u-u_{B_r}}{r}\right|^p+ a(x)  \left|\dfrac{u-u_{B_r}}{r}\right|^q \right)^{d_1} \,\dd\mu\right)^{\frac{1}{d_1}} \\
	&\qquad\leq \left(\dashint_{B_r}\left(\left|\dfrac{u-u_{B_r}}{r}\right|^p
	+  C[a]_{\alpha}\mu(B_{2\lambda r})^{\alpha/Q}\left|\dfrac{u-u_{B_r}}{r}\right|^q \right)^{d_1} \,\dd\mu\right)^{\frac{1}{d_1}}  \\
	&\qquad\leq \left(\dashint_{B_r}\left\vert\frac{u-u_{B_r}}{r}\right\vert^{pd_1}\,\dd\mu\right)^{\frac{1}{d_1}}
	+C[a]_{\alpha}\mu(B_{2\lambda r})^{\alpha/Q}\left(\dashint_{B_r}\left\vert\frac{u-u_{B_r}}{r}\right\vert^{q\, d_1}\,\dd\mu\right)^{\frac{1}{d_1}}\\
	 &\qquad\leq C\left(\dashint_{B_r}\left\vert\frac{u-u_{B_r}}{r}\right\vert^{qd_1}\,\dd\mu\right)^{\frac{1}{d_1}\frac{p}{q}} 
    \left(1+[a]_{\alpha}\mu(B_{2\lambda r})^{\alpha/Q}\left(\dashint_{B_r}\left\vert\frac{u-u_{B_r}}{r}\right\vert^{qd_1}\,\dd\mu\right)^{\frac{q-p}{qd_1}}\right),
\end{split}
\end{equation}
where $C=C(C_D,\alpha)$.
Since $qd_1<s^*$ and $s<p$,  \eqref{spoincare} and H\"older's inequality imply
\[
\left(\dashint_{B_r}\left\vert \frac{u-u_{B_r}}{r}\right\vert^{qd_1}\,\dd\mu\right)^{\frac{1}{qd_1}}
\leq C\left(\dashint_{B_{2\lambda r}}g_u^{p}\,\dd\mu\right)^{\frac{1}{p}},
\]
where $C=C(\mathrm{data})$.
Thus we have
\[
\left(\dashint_{B_r}\left\vert\frac{ u-u_{B_r}}{r}\right\vert^{qd_1}\,\dd\mu\right)^{\frac{q-p}{qd_1}}
\leq C\left(\dashint_{B_{2\lambda r}}g_u^{p}\,\dd\mu\right)^{\frac{q-p}{p}}
=C\Vert g_u\Vert^{q-p}_{L^p(B_{2\lambda r})}\mu(B_{2\lambda r})^{-\frac{q-p}{p}},
\]
where $C= C(\mathrm{data})$.
By \eqref{6lemma3} we obtain
\[
\left(\dashint_{B_r} H\left(x,\frac{u-u_{B_r}}{r}\right)^{d_1}\,\dd\mu\right)^{\frac{1}{d_1}}
\leq C\left(\dashint_{B_r}\left\vert\frac{u-u_{B_r}}{r}\right\vert^{qd_1}\,\dd\mu\right)^{\frac{1}{d_1}\frac{p}{q}}
 \left(1+\Vert g_u\Vert^{q-p}_{L^p(B_{2\lambda r})}\mu(B_{2\lambda r})^{\frac{\alpha}{Q}-\frac{q-p}{p}}\right),
\]
where $C= C(\mathrm{data})$. 
Since $qd_1<s^*$ and $s<pd_2$, by \eqref{spoincare} we have
\begin{align*}
\left(\dashint_{B_r}\left\vert\frac{u-u_{B_r}}{r}\right\vert^{qd_1}\,\dd\mu\right)^{\frac{1}{d_1}\frac{p}{q}}
&\leq C\left(\dashint_{B_{2\lambda r}}g_u^{p\, d_2}\,\dd \mu\right)^{\frac{1}{d_2}}\\
&\leq C \left(\dashint_{B_{2\lambda r}} \left(g_u^p+a(x)g_u^q\right)^{d_2}\,\dd\mu\right)^{\frac{1}{d_2}}\\
& = C\left(\dashint_{B_{2\lambda r}} H(x,g_u)^{d_2}\,\dd\mu\right)^{\frac{1}{d_2}}.
\end{align*}
This completes the proof.
\end{proof}

Next we discuss a Sobolev inequality for functions which vanish on a large set, see \cite{KS}.	

\begin{lemma}\label{lemma2density} 
Assume that $u\in N^{1,1}_{\loc}({\Omega})$ with $H(\cdot,g_u)\in L^1_{\loc}(\Omega)$. 
Let $B_r$ be a ball and $a_0=\inf_{x\in B_{2\lambda r}}a(x)$. Assume that there exists $\gamma$, $0<\gamma<1$, such that
\[
\mu(\lbrace x\in B_r:\vert u(x)\vert>0\rbrace)\leq\gamma\mu(B_r).
\]
Then there exist a constant $C= C(\mathrm{data},\gamma)$ and exponents $0<d_2<1\leq d_1<\infty$, with $d_1= d_1(\mathrm{data})$ and $d_2= d_2(\mathrm{data})$, such that
\[
\left(\dashint_{B_r}\left( \left|\dfrac{u}{r}\right|^p+a_0\left|\dfrac{u}{r}\right|^q \right)^{d_1}\,\dd\mu\right)^{\frac{1}{d_1}} 
\leq C\left(\dashint_{B_{2\lambda r}} \left( g_u^{p}+ a_0 g_u^q\right)^{d_2}\,\dd\mu\right)^{\frac{1}{d_2}}.
\]
\end{lemma}

\begin{proof} 
As in the proof of Lemma \ref{Lemma2}, there exists $s$, with $1<s<p<q<s^*$, such that \eqref{spoincare} holds.
Let $A=\lbrace x\in B_r:\vert u(x)\vert>0\rbrace$. 
We observe that
\begin{equation}\label{eqlemma1boundary}
    \left(\dashint_{B_r}\left\vert\frac{u}{r}\right\vert^{s^*}\,\dd\mu\right)^{\frac{1}{s^*}}
    \leq \left(\dashint_{B_r} \left|\dfrac{u-u_{B_r}}{r}\right|^{s^*}\,\dd\mu\right)^{\frac{1}{s^*}}+\left|\dfrac{u_{B_r}}{r}\right|.
\end{equation}
By H\"older's inequality we obtain
\begin{equation*}
    \vert u_{B_r}\vert
    \le\frac{1}{\mu(B_r)}\int_{A}\vert u\vert\,\dd\mu
    \leq\left(\frac{\mu(A)}{\mu(B_r)}\right)^{1-\frac{1}{s^*}}\left(\dashint_{B_r}\vert u\vert^{s^*}\,\dd\mu\right)^{\frac{1}{s^*}}
    \leq \gamma^{1-\frac{1}{s^*}}\left(\dashint_{B_r}\vert u\vert^{s^*}\,\dd\mu\right)^{\frac{1}{s^*}}.
\end{equation*}
By \eqref{eqlemma1boundary} and \eqref{spoincare} we conclude that
\begin{equation}\label{pzpoincare}
    (1-\gamma^{1-\frac{1}{s^*}})\left(\dashint_{B_r}\left\vert \frac{u}{r}\right\vert^{s^*}\,\dd\mu\right)^{\frac{1}{s^*}}
    \leq\left(\dashint_{B_r} \left|\dfrac{u-u_{B_r}}{r}\right|^{s^*}\,\dd\mu\right)^{\frac{1}{s^*}}
    \leq C\left(\dashint_{B_{2\lambda r}} g_u^{s}\,\dd\mu\right)^{\frac{1}{s}},
\end{equation}
where $C= C(\mathrm{data})$.
Since $pd_1<qd_1<s^*$ and $s<pd_2<qd_2$, by H\"older's inequality, we have
\[
\left(\dashint_{B_r}\left\vert\frac{u}{r}\right\vert^{pd_1}\dd\mu\right)^{\frac{1}{pd_1}}
\leq C\left(\dashint_{B_{2\lambda r}} g_u^{pd_2}\dd\mu\right)^{\frac{1}{pd_2}},
\]
and
\[
\left(\dashint_{B_r}\left\vert\frac{u}{r}\right\vert^{qd_1}\dd\mu\right)^{\frac{1}{qd_1}}
\leq C\left(\dashint_{B_{2\lambda r}} g_u^{qd_2}\dd\mu\right)^{\frac{1}{qd_2}},
\]
where $C= C(\mathrm{data},\gamma)$. 
The rest of the proof follows as in the proof of Lemma \ref{Lemma2}.
\end{proof}

Then we consider a local double phase Sobolev--Poincar\'e inequality.

\begin{lemma}\label{lemma3cap}
Assume that $u\in N^{1,1}_{\loc}({\Omega})$ with $H(\cdot,g_u)\in L^1_{\loc}(\Omega)$.
Let $B_r$ be a ball and assume that there exist $\gamma$, $0<\gamma<1$, such that
\[
\mu(\lbrace x\in B_r:\vert u(x)\vert>0\rbrace)\leq\gamma\mu(B_r).
\]
Then there exists a constant $C=C(\mathrm{data},\gamma)$ and exponents $0<d_2<1\leq d_1<\infty$,  with $d_1= d_1(\mathrm{data})$ and $d_2= d_2(\mathrm{data})$, such that 
\begin{equation*}
\left(\dashint_{B_r} H\left(x,\frac{u}{r}\right)^{d_1}\,\dd\mu\right)^{\frac{1}{d_1}}
\leq C\left(1+\Vert g_u\Vert_{L^p(B_{2\lambda r})}^{q-p}\mu(B_{2\lambda r})^{\frac{\alpha}{Q}-\frac{q-p}{p}}\right)
\left(\dashint_{B_{2\lambda r}} H(x,g_u)^{d_2}\,\dd\mu\right)^{\frac{1}{d_2}}.
\end{equation*}
\end{lemma}

\begin{proof}
As in the proof of Lemma \ref{Lemma3}, we consider two cases \eqref{1C}  and \eqref{3C}.
If \eqref{1C} holds, then as in \eqref{2C} with $\frac{u}{r}$ instead of $\frac{u-u_{B_r}}{r}$ and Lemma \ref{lemma2density}, we obtain 
\begin{equation*}
\left(\dashint_{B_r} H\left(x,\frac{u}{r}\right)^{d_1}\,\dd\mu\right)^{\frac{1}{d_1}}
\leq C \left(\dashint_{B_{2\lambda r}} H(x,g_u)^{d_2}\,\dd\mu\right)^{\frac{1}{d_2}},
\end{equation*}
where $C= C(\mathrm{data},\gamma)$, $d_1= d_1(\mathrm{data})$ and $d_2= d_2(\mathrm{data})$ with $0<d_2<1\leq d_1<\infty$.
On the other hand, if \eqref{3C} holds, then as in \eqref{6lemma3}, we obtain
\begin{align*}
    \left(\dashint_{B_r}H\left(x,\frac{u}{r}\right)^{d_1}\,\dd\mu\right)^{\frac{1}{d_1}}
    &\leq\left(\dashint_{B_r}\left(\left\vert\frac{u}{r}\right\vert^p+C[a]_{\alpha}\mu(B_{2\lambda r})^{\alpha/Q}\left\vert\frac{u}{r}\right\vert^q\right)^{d_1}\,\dd\mu\right)^{\frac{1}{d_1}}\\
    &\leq\left(\dashint_{B_r}\left\vert\frac{u}{r}\right\vert^{pd_1}\,\dd\mu\right)^{\frac{1}{d_1}}
    +C[a]_{\alpha}\mu(B_{2\lambda r})^{\alpha/Q}\left(\dashint_{B_r}\left\vert\frac{u}{r}\right\vert^{qd_1}\,\dd\mu\right)^{\frac{1}{d_1}}\\
    &\leq C\left(\dashint_{B_r}\left\vert\frac{u}{r}\right\vert^{qd_1}\,\dd\mu\right)^{\frac{p}{qd_1}}\left(1+[a]_{\alpha}\mu(B_{2\lambda r})^{\alpha/Q}\left(\dashint_{B_r}\left\vert\frac{u}{r}\right\vert^{qd_1}\,\dd\mu\right)^{\frac{q-p}{qd_1}}\right),
\end{align*}
where $C=C(C_D,\alpha)$.
Since $qd_1<s^*$ and $s<pd_2$,  \eqref{pzpoincare} and H\"older's inequality imply
\[
\left(\dashint_{B_r}\left\vert\frac{u}{r}\right\vert^{qd_1}\dd\mu\right)^{\frac{1}{qd_1}}
\leq C\left(\dashint_{B_{2\lambda r}} g_u^{p}\,\dd\mu\right)^{\frac{1}{p}},
\]
 where $C= C(\mathrm{data},\gamma)$. 
As in the proof of Lemma \ref{Lemma3}, we have
\begin{align*}
    \left(\dashint_{B_r}\left\vert\frac{u}{r}\right\vert^{qd_1}\,\dd\mu\right)^{\frac{q-p}{qd_1}}
    &\leq C\left(\dashint_{B_{2\lambda r}}g_u^p\,\dd\mu\right)^{\frac{q-p}{p}}
    =C\Vert g_u\Vert_{L^p(B_{2\lambda r})}^{q-p}\mu(B_{2\lambda r})^{-\frac{q-p}{p}},
\end{align*}
and thus
\begin{align*}
     \left(\dashint_{B_r}H\left(x,\frac{u}{r}\right)^{d_1}\,\dd\mu\right)^{\frac{1}{d_1}}
     &\leq C\left(\dashint_{B_r}\left\vert\frac{u}{r}\right\vert^{qd_1}\,\dd\mu\right)^{\frac{p}{q\, d_1}}
     \left(1+\Vert g_u\Vert_{L^p(B_{2\lambda r})}^{q-p}\mu(B_{2\lambda r})^{\frac{\alpha}{Q}-\frac{q-p}{p}}\right),
       \end{align*}
where $C= C(\mathrm{data},\gamma)$. 
Since $qd_1<s^*$ and $s<pd_2$, applying \eqref{pzpoincare} as in the proof of Lemma \ref{lemma2density}, we have
\begin{align*}
\left(\dashint_{B_r}\left\vert\frac{u}{r}\right\vert^{qd_1}\,\dd\mu\right)^{\frac{p}{qd_1}}
     &\leq C\left(\dashint_{B_{2\lambda r}}g_u^{pd_2}\,\dd \mu\right)^{\frac{1}{d_2}}\\
     &\leq C \left(	\dashint_{B_{2\lambda r}} \left(g_u^p+a(x)g_u^q\right)^{d_2}\,\dd\mu\right)^{\frac{1}{d_2}}\\
     & = C\left(\dashint_{B_{2\lambda r}} H(x,g_u)^{d_2}\,\dd\mu\right)^{\frac{1}{d_2}},
\end{align*}
where $C= C(\mathrm{data},\gamma)$. This completes the proof.
\end{proof}

\section{Local and global higher integrability results}\label{global}

The main goal of this work is to get global higher integrability for quasiminimizers.
In the metric setting, the improvement of integrability is obtained by using a metric space version of Gehring's lemma, whose proof can be found, for example, in \cite{BB} or \cite{ZG}. 

\begin{lemma}\label{Gehring}
 Let $f\in L^1_{\loc}(\Omega)$ and $g\in L^{\sigma}_{\loc}(X)$, $\sigma>1$, be non-negative functions and let $\lambda>1$. 
 Assume that there exist a constant $C_1$ and an exponent $0<d<1$ such that
	\begin{equation*}
		\dashint_{B_R}f \,\dd\mu
		\leq C_1\left(\left(\dashint_{B_{\lambda R}}f^d \,\dd\mu\right)^{\frac{1}{d}}+\dashint_{B_{\lambda R}}g \,\dd\mu\right),
	\end{equation*}
for every ball $B_R$ with $B_{\lambda R}\Subset\Omega$. 
Then there exist a constant $C_2=C_2(C_D,C_1,d,\lambda)$ and an exponent $\varepsilon=\varepsilon(C_D,C_1,d,\lambda)>0$ such that
	\begin{equation*}
		\left(\dashint_{B_R}f^{1+\varepsilon} \,\dd\mu\right)^{\frac{1}{1+\varepsilon}}
		\leq C_2\left(\dashint_{B_{\lambda R}}f\,\dd\mu
		+\left(\dashint_{B_{\lambda R}}g^{\sigma} \,\dd\mu\right)^{\frac{1}{\sigma}}\right),
	\end{equation*}
for every ball $B_R$ with $B_{\lambda R}\Subset\Omega$. 
\end{lemma}

Next we discuss local higher integrability of the upper gradient of a local quasiminimizer.

\begin{theorem}\label{mainlocal}
Let $\Omega$ be an open subset of $X$ and let  $\Omega'\Subset\Omega''\Subset\Omega$.
Assume that $u\in N^{1,1}_{\loc}({\Omega})$ with $H(\cdot,g_u)\in L^1_{\loc}(\Omega)$ is a local quasiminimizer in $\Omega$. 
Then there exist a constant $C=C(\mathrm{data},\Omega'',\Vert g_u\Vert_{L^p(\Omega'')})$ and an exponent $\varepsilon=\varepsilon(\mathrm{data},\Omega'',\Vert g_u\Vert_{L^p(\Omega'')})>0$ such that 
\[
 \left(\int_{\Omega'}H(x,g_u)^{1+\varepsilon}\,\dd\mu\right)^{\frac{1}{1+\varepsilon}}
\leq C\int_{\Omega''} H(x,g_u)\,\dd\mu.
\]
\end{theorem}

\begin{proof}
Let $B_{2\lambda r}\Subset\Omega''$. By Lemma \ref{DeGiorgiLemma1}, there exists a constant $C=C(K,q)$ such that
\[
\int_{B_{\frac r2}}H(x,g_u)\,\dd\mu
\leq C\int_{B_r} H\left(x,\frac{u-u_{B_r}}{r}\right)\,\dd\mu.
\]
On the other hand, by Lemma \ref{Lemma3}, we obtain 
\[
\dashint_{B_r} H\left(x,\frac{u-u_{B_r}}{r}\right)\,\dd\mu
\leq C\left(1+\Vert g_u\Vert^{q-p}_{L^p(B_{2\lambda r})}\mu(B_{2\lambda r})^{\frac{\alpha}{Q}-\frac{q-p}{p}}\right)
 \left(\dashint_{B_{2\lambda r}} H(x,g_u)^{d}\,\dd\mu\right)^{\frac{1}{d}},
\]
where $0<d=d(\mathrm{data})<1$ and $C=C(\mathrm{data})$.
This implies that
\[
\dashint_{B_{\frac r2}}H(x,g_u)\,\dd \mu
\leq C\left(1+\Vert g_u\Vert^{q-p}_{L^p(B_{2\lambda r})}\mu(B_{2\lambda r})^{\frac{\alpha}{Q}-\frac{q-p}{p}}\right)
 \left(\dashint_{B_{2\lambda r}}H(x,g_u)^{d}\,\dd\mu\right)^{\frac{1}{d}},
\]
where $C=C(\mathrm{data})$.

By \eqref{pqcond} we have $\frac{\alpha}{Q}-\frac{q-p}{p}\ge0$ and thus we obtain
\[
\dashint_{B_{r/2}}H(x,g_u)\, \dd \mu
\leq C\left(1+\Vert g_u\Vert^{q-p}_{L^p(\Omega'')}\mu(\Omega'')^{\frac{\alpha}{Q}-\frac{q-p}{p}}\right)
 \left(\dashint_{B_{2\lambda r}} H(x,g_u)^{d}\,\dd\mu\right)^{\frac{1}{d}},
\]
where  $C=C(\mathrm{data})$.
This implies that
\begin{equation}\label{drh}
\dashint_{B_{r/2}}H(x,g_u)\, \dd \mu
\leq C\left(\dashint_{B_{2\lambda r}} H(x,g_u)^{d}\,\dd\mu\right)^{\frac{1}{d}},
\end{equation}
for every ball with $B_{2\lambda r}\Subset\Omega''$ with $C=C(\mathrm{data},\Omega'',\Vert g_u\Vert_{L^p(\Omega'')})$.
The constant $C$ depends on $\Omega''$ and on $\Vert g_{u}\Vert_{L^p(\Omega'')}$, but once 
$u$ and $\Omega''$ are fixed, the obtained reverse H\"older inequality is uniform over all balls with $B_{2\lambda r}\Subset\Omega''$. 
By Lemma \ref{Gehring}, there exist a constant $C=C(\mathrm{data},\Omega'',\Vert g_u\Vert_{L^p(\Omega'')})$ and an exponent  $\varepsilon=\varepsilon(\mathrm{data},\Omega'',\Vert g_u\Vert_{L^p(\Omega'')})>0$ such that
\[
\left(\dashint_{B_{r/2}} H(x,g_u)^{1+\varepsilon}\,\dd\mu\right)^{\frac{1}{1+\varepsilon}}
\le C\dashint_{B_{2\lambda r}}H(x,g_u)\, \dd \mu,
\]
for every ball with $B_{2\lambda r}\Subset\Omega''$. 
Since $\overline{\Omega'}$ is compact, we can cover it by a finite number of such balls and conclude that 
\[
 \left(\int_{\Omega'}H(x,g_u)^{1+\varepsilon}\,\dd\mu\right)^{\frac{1}{1+\varepsilon}}
\leq C\int_{\Omega''} H(x,g_u)\,\dd\mu.
\]
\end{proof}

\begin{remark}
If the measure is Ahlfors--David regular, see \eqref{ahlfors}, we have
\[
\mu(B_{2\lambda r})^{\frac{\alpha}{Q}-\frac{q-p}{p}}
\le Cr^{\alpha-Q(\frac qp-1)},
\]
where $\alpha-Q(\frac qp-1)\ge0$ by  \eqref{pqcond}.
As in the proof of Theorem \ref{mainlocal}, there exist a constant $C=C(\mathrm{data})$ and an exponent $0<d=d(\mathrm{data})<1$ such that
\[
\dashint_{B_{\frac r2}}H(x,g_u)\,\dd \mu
\leq C\left(1+\Vert g_u\Vert^{q-p}_{L^p(B_{2\lambda r})}\right)
 \left(\dashint_{B_{2\lambda r}}H(x,g_u)^{d}\,\dd\mu\right)^{\frac{1}{d}},
\]
whenever $B_{2\lambda r}\Subset\Omega$ with $0<r\le 1$.
A similar argument can also be applied in Lemma \ref{Lemma3} and Lemma \ref{lemma3cap}.
 \end{remark}

Finally, we are ready to prove the main result of the paper, which states higher integrability for the weak upper gradient of a quasiminimizer over the entire domain
under the assumption that the domain satisfies a uniform measure density property.

\begin{theorem}\label{mainglobal}
Assume that $\Omega$ is a bounded open set in $X$ with the property that there exists a constant $\gamma$, $0<\gamma<1$, for which
\[
\mu(B_{R}(x)\cap\Omega)\leq\gamma\mu(B_{R}(x)),
\]
for every $x\in X\setminus\Omega$ and $R>0$.
Assume that $w\in N^{1,1}(\Omega)$ such that $H(\cdot, g_w)\in L^{\sigma}(\Omega)$ for some $\sigma>1$. 
Assume that $u\in N^{1,1}({\Omega})$ with $H(\cdot,g_u)\in L^1(\Omega)$ is a quasiminimizer in $\Omega$ with $u-w\in N^{1,1}_0(\Omega)$.
Then there exist a constant $C=C(\mathrm{data},\gamma,\Omega,\Vert g_{u-w}\Vert_{L^p(\Omega)})$ and an exponent $\varepsilon=\varepsilon(\mathrm{data},\gamma,\Omega,\Vert g_{u-w}\Vert_{L^p(\Omega)})>0$ such that
\[
\left(\int_{\Omega}H(x,g_u)^{1+\varepsilon}\ \dd\mu\right)^{\frac{1}{1+\varepsilon}}
\leq C\left(\int_{\Omega}H(x,g_u)\ \dd\mu+\left(\int_{\Omega}H(x,g_w)^{\sigma}\ \dd\mu\right)^{\frac{1}{\sigma}}\right).
\]
\end{theorem}

\begin{proof}
Let $B_r$ be a ball with $B_r\cap\Omega\ne\emptyset$ and $0<r\le1$.
Then there exist two alternatives: either $B_{3\lambda r}\subset\Omega$ or $B_{3\lambda r}\setminus\Omega\ne\emptyset$.
If $B_{3\lambda r}\subset\Omega$, then $B_{2\lambda r}\Subset\Omega$ and, as in \eqref{drh}, we have
\begin{equation}\label{rhlocal}
\dashint_{B_{r/2}}H(x,g_u)\, \dd \mu
\leq C\left(\dashint_{B_{2\lambda r}} H(x,g_u)^{d}\,\dd\mu\right)^{\frac{1}{d}},
\end{equation}
where $C=C(\mathrm{data},\Omega,\Vert g_u\Vert_{L^p(\Omega)})$.

Then we discuss the case $B_{3\lambda r}\setminus\Omega\ne\emptyset$.
Let $x_0\in B_{3\lambda r}\setminus\Omega$ and consider $B_{R}(x_0)$ with $R=8\lambda r$.
Since the center of $B_r$ is contained in $B_{3\lambda r}(x_0)$, we have $B_r\subset B_{4\lambda r}(x_0)=B_{\frac R2}(x_0)$.
 Let $B_R=B_{R}(x_0)$.
We note that $B_{3\lambda r}\subset B_{R}$ and $\mu(B_{R}\cap\Omega)\leq\gamma\mu(B_{R})$,
with $0<\gamma<1$.
Since $u-w=0$ $\mu$-almost everywhere in $X\setminus\Omega$, we obtain 
\[
\mu(\lbrace x\in B_{R}: \vert u(x)-w(x)\vert>0\rbrace)\leq\gamma\mu(B_R).
\] 

By Lemma \ref{DeGiorgiLemma2} there exists a constant $C=C(K,q)$ such that
 \begin{equation}\label{suneq}
  \int_{B_{\frac R2}\cap\Omega}H(x, g_u)\,\dd\mu
  \leq C\left( \int_{B_R\cap\Omega}H\left(x, \frac{u-w}{R}\right)\,\dd\mu
  +\int_{B_R\cap\Omega}H(x, g_w)\,\dd\mu\right).
\end{equation}
We consider the first term on the right-hand side of \eqref{suneq}.
Since $u-w=0$ $\mu$-almost everywhere in $X\setminus\Omega$ and $g_{u-w}=0$ $\mu$-almost everywhere in $X\setminus\Omega$, by applying Lemma \ref{lemma3cap} with $u-w$, we obtain
\begin{align*}
   &\frac{1}{\mu(B_R)}\int_{B_R\cap\Omega}H\left(x, \frac{u-w}{R}\right)\,\dd\mu
   = \dashint_{B_R}H\left(x, \frac{u-w}{R}\right)\,\dd\mu\\
  &\qquad\leq  C\left(1+\Vert g_{u-w}\Vert_{L^p(B_{2\lambda R})}^{q-p}\mu(B_{2\lambda R})^{\frac{\alpha}{Q}-\frac{q-p}{p}}\right)
  \left(\dashint_{B_{2\lambda R}}H(x,g_{u-w})^d\,\dd\mu\right)^{\frac{1}{d}},
\end{align*}
where $C=C(\mathrm{data},\gamma)$ and $0<d= d(\mathrm{data})<1$.
Since $B_r\subset B_{\frac R2}(x_0)$, $B_r\cap\Omega\ne\emptyset$ and $0<r\le1$, 
we have $B_{R}\cap\Omega\ne\emptyset$ and $2\lambda R=16\lambda^2r\le16\lambda^2$.
This implies that
\[
B_{2\lambda R}\subset\Omega^*=\{y\in X:\dist(y,\Omega)<24\lambda^2\}.
\]
By \eqref{pqcond} we have $\frac{\alpha}{Q}-\frac{q-p}{p}\ge0$ and thus we obtain
\[
\frac{1}{\mu(B_R)}\int_{B_R\cap\Omega}H\left(x, \frac{u-w}{R}\right)\,\dd\mu
\leq  C\left(1+\Vert g_{u-w}\Vert_{L^p(\Omega)}^{q-p}\mu(\Omega^*)^{\frac{\alpha}{Q}-\frac{q-p}{p}}\right)
  \left(\dashint_{B_{2\lambda R}}H(x,g_{u-w})^d\,\dd\mu\right)^{\frac{1}{d}}.
\]
This implies that
\begin{align*}
\frac{1}{\mu(B_R)}\int_{B_R\cap\Omega}H\left(x, \frac{u-w}{R}\right)\,\dd\mu
&\leq C\left(\dashint_{B_{2\lambda R}} H(x,g_{u-w})^{d}\,\dd\mu\right)^{\frac{1}{d}}\\
&=C\left(\frac{1}{\mu(B_{2\lambda R})}\int_{B_{2\lambda R}\cap\Omega} H(x,g_{u-w})^{d}\,\dd\mu\right)^{\frac{1}{d}},
\end{align*}
where $C=C(\mathrm{data},\gamma,\Omega,\Vert g_{u-w}\Vert_{L^p(\Omega)})$.
Thus we have
\begin{equation}\label{spiralequation}
\begin{split}
&\frac{1}{\mu(B_R)}\int_{B_R\cap\Omega}H\left(x, \frac{u-w}{R}\right)\,\dd\mu
\leq C\left(\frac{1}{\mu(B_{2\lambda R})}\int_{B_{2\lambda R}\cap\Omega} H(x,g_{u-w})^{d}\,\dd\mu\right)^{\frac{1}{d}}\\
   &\qquad\leq C\left(\frac{1}{\mu(B_{2\lambda R})}\int_{B_{2\lambda R}\cap\Omega}H(x,g_{u}+g_w)^d\,\dd\mu\right)^{\frac{1}{d}}\\
    &\qquad\leq C\left(\frac{1}{\mu(B_{2\lambda R})}\int_{B_{2\lambda R}\cap\Omega}\left( H(x,g_u)^d+H(x,g_w)^d\right)\,\dd\mu\right)^{\frac{1}{d}}\\
    &\qquad\leq C\left(\left(\frac{1}{\mu(B_{2\lambda R})}\int_{B_{2\lambda R}\cap\Omega} H(x,g_u)^d\,\dd\mu\right)^{\frac{1}{d}}
    +\left(\frac{1}{\mu(B_{2\lambda R})}\int_{B_{2\lambda R}\cap\Omega} H(x,g_w)^d\,\dd\mu\right)^{\frac{1}{d}}\right)\\
     &\qquad\leq C\left(\left(\frac{1}{\mu(B_{2\lambda R})}\int_{B_{2\lambda R}\cap\Omega} H(x,g_u)^d\,\dd\mu\right)^{\frac{1}{d}}
     +\frac{1}{\mu(B_{2\lambda R})}\int_{B_{2\lambda R}\cap\Omega} H(x,g_w)\,\dd\mu\right),
     \end{split}
\end{equation}
where $C=C(\mathrm{data},\gamma,\Omega,\Vert g_{u-w}\Vert_{L^p(\Omega)})$.
By \eqref{suneq}, \eqref{spiralequation} and the doubling property \eqref{doubling}, we obtain
\begin{align*}
 &\frac{1}{\mu(B_{\frac R2})}\int_{B_{\frac R2}\cap\Omega}H(x,g_u)\,\dd\mu\\
  &\qquad\leq C\left(\frac{1}{\mu(B_R)}\int_{B_R\cap\Omega}H\left(x, \frac{u-w}{R}\right)\,\dd\mu
 +\frac{1}{\mu(B_{2\lambda R})}\int_{B_{{2\lambda R}}\cap\Omega}H(x,g_w)\ \dd\mu \right)\\
 &\qquad\leq C\left(\left(\frac{1}{\mu(B_{2\lambda R})}\int_{B_{2\lambda R}\cap\Omega}H(x,g_u)^{d}\,\dd\mu\right)^\frac{1}{d}
 + \frac{1}{\mu(B_{2\lambda R})}\int_{B_{2\lambda R}\cap\Omega}H(x,g_w)\,\dd\mu\right),
\end{align*}
where $C=C(\mathrm{data},\gamma,\Omega,\Vert g_{u-w}\Vert_{L^p(\Omega)})$.

Let $f=H(x,g_u)\chi_\Omega$ and $g=H(x,g_w)\chi_\Omega$.
Since $B_{\frac r2}\subset B_r\subset B_{\frac R2}$, we obtain
\begin{align*}
\dashint_{B_{\frac r2}}f\,\dd\mu
&\le C\dashint_{B_{\frac R2}}f\,\dd\mu
\leq C\left(\left(\dashint_{B_{2\lambda R}}f^d\,\dd\mu\right)^{\frac{1}{d}}
+\dashint_{B_{2\lambda R}}g\,\dd\mu\right)\\
&=C\left(\left(\dashint_{B_{16\lambda^2 r}}f^d\,\dd\mu\right)^{\frac{1}{d}}
+\dashint_{B_{16\lambda^2 r}}g\,\dd\mu\right),
\end{align*}
whenever  $B_{3\lambda r}\setminus\Omega\ne\emptyset$.
If $B_{3\lambda r}\subset\Omega$, by \eqref{rhlocal} we have
\[
\dashint_{B_{\frac r2}}f\,\dd\mu
\leq C\left(\dashint_{B_{2\lambda r}}f^d\,\dd\mu\right)^{\frac{1}{d}}
\le C\left(\left(\dashint_{B_{16\lambda^2 r}}f^d\,\dd\mu\right)^{\frac{1}{d}}
+\dashint_{B_{16\lambda^2 r}}g\,\dd\mu\right).
\]
It follows that
\[
\dashint_{B_{\frac r2}}f\,\dd\mu
\le C\left(\left(\dashint_{B_{16\lambda^2 r}}f^d\,\dd\mu\right)^{\frac{1}{d}}
+\dashint_{B_{16\lambda^2 r}}g\,\dd\mu\right)
\]
for every ball $B_r$ in $X$ with $0<r\le1$.
Note that this is trivially true for balls with $B_r\cap\Omega=\emptyset$, since the left-hand side is zero.
By a straight forward covering argument we have
\[
\dashint_{B_{\frac r2}}f\,\dd\mu
\le C\left(\left(\dashint_{B_r}f^d\,\dd\mu\right)^{\frac{1}{d}}
+\dashint_{B_r}g\,\dd\mu\right)
\]
for every ball $B_r$ in $X$ with $0<r\le1$.
Here $C=C(\mathrm{data},\Omega,\gamma,\Vert g_{u-w}\Vert_{L^p(\Omega)})$ and $0<d=d(\mathrm{data})<1$.
Note carefully, that the constant $C$ depends on the underlying set $\Omega$ and on $\Vert g_{u-w}\Vert_{L^p(\Omega)}$, but once the domain
$\Omega$ and the boundary function $w$ are fixed, the obtained reverse H\"older inequality is uniform over balls $B_r$ in $X$ with $0<r\le1$.
By an application of Lemma \ref{Gehring}, there exist a constant $C=C(\mathrm{data},\gamma,\Omega,\Vert g_{u-w}\Vert_{L^p(\Omega)})$ and an exponent $\varepsilon=\varepsilon(\mathrm{data},\Omega,\gamma,\Vert g_{u-w}\Vert_{L^p(\Omega)})>0$ such that
$$
\left(\dashint_{B_{\frac r2}}f^{1+\varepsilon}\,\dd\mu\right)^{\frac{1}{1+\varepsilon}}
\leq C\left(\dashint_{B_r}f\,\dd\mu
+\left(\dashint_{B_r}g^{\sigma}\,\dd\mu\right)^{\frac{1}{\sigma}}\right),
$$
for every ball $B_r$ in $X$ with $0<r\le1$.
Thus, we have
\begin{equation*}
\begin{split}
&\left(\frac{1}{\mu(B_{\frac r2})}\int_{B_{\frac r2}\cap \Omega}H(x,g_u)^{1+\varepsilon}\,\dd\mu\right)^{\frac{1}{1+\varepsilon}}\\
&\qquad\leq C\left(\frac{1}{\mu(B_r)}\int_{B_r\cap \Omega}H(x,g_u)\,\dd\mu
+\left(\frac{1}{\mu(B_r)}\int_{B_r\cap \Omega}H(x,g_w)^{\sigma}\,\dd\mu\right)^{\frac{1}{\sigma}}\right),
\end{split}
\end{equation*}
for every ball with $B_r$ in $X$ with $0<r\le1$.
Since $\Omega$ is bounded, we may cover it by a finite number of balls $B_{r_j}(x_j)$, $j=1,2,\dots,N$, with $0<r_j\le1$.
By summing over $j=1,\dots,N$, we obtain
\[
\left(\int_{\Omega}H(x,g_u)^{1+\varepsilon}\ \dd\mu\right)^{\frac{1}{1+\varepsilon}}
\leq C\left(\int_{\Omega}H(x,g_u)\ \dd\mu+\left(\int_{\Omega}H(x,g_w)^{\sigma}\ \dd\mu\right)^{\frac{1}{\sigma}}\right),
\]
where $C=C(\mathrm{data},\gamma,\Omega,\Vert g_{u-w}\Vert_{L^p(\Omega)})$.
\end{proof}

\section*{Acknowledgements}
Part of this material is based upon work supported by the Swedish Research Council while the first author and the second author were in residence at Institut Mittag-Leffler in Djursholm, Sweden during the Research Program Geometric Aspects of Nonlinear Partial Differential Equations in 2022.

The second author is a member of the Gruppo Nazionale per l'Analisi Matematica, la Probabilit\`{a} e le loro Applicazioni (GNAMPA) of the Istituto Nazionale di Alta Matematica (INdAM).
The second author was partly supported by  GNAMPA-INdAM  Project 2022 "Equazioni differenziali alle derivate parziali in fenomeni non lineari"
and by  GNAMPA-INdAM  Project 2023 "Regolarità per problemi ellittici e parabolici con crescite non standard". 

The third author was supported by a doctoral training grant for 2022 and a travel grant from the V\"ais\"al\"a Fund.

\end{document}